\documentclass[oneside,english]{amsart}
\usepackage[T1]{fontenc}
\usepackage[latin9]{inputenc}
\usepackage{amstext}
\usepackage{amsthm}
\usepackage{graphicx}
\PassOptionsToPackage{normalem}{ulem}
\usepackage{ulem}

\makeatletter
\numberwithin{equation}{section}
\numberwithin{figure}{section}
\theoremstyle{plain}
\newtheorem{thm}{\protect\theoremname}
\theoremstyle{definition}
\newtheorem{defn}[thm]{\protect\definitionname}
\theoremstyle{plain}
\newtheorem{lem}[thm]{\protect\lemmaname}
\theoremstyle{plain}
\newtheorem{prop}[thm]{\protect\propositionname}

\makeatother

\usepackage{babel}
\providecommand{\definitionname}{Definition}
\providecommand{\lemmaname}{Lemma}
\providecommand{\propositionname}{Proposition}
\providecommand{\theoremname}{Theorem}

\begin{document}
\title[Orientation-preserving and orientation-reversing mappings]{Orientation-preserving and orientation-reversing mappings: a new description}
\author{Peter M. Higgins and Alexei Vernitski}
\begin{abstract}
We characterise the respective semigroups of mappings that preserve,
or that preserve or reverse orientation of a finite cycle, in terms
of their actions on oriented triples and oriented quadruples. This
leads to a proof that the latter semigroup coincides with the semigroup
of all mappings that preserve intersections of chords on the corresponding
circle.
\end{abstract}

\maketitle

\section{Orientation-preserving and orientation-reversing mappings on a cycle }

This section presents definitions and some known results; it is based
mainly on \cite{CH1999} and also on \cite{M1998}. Let $[n]$ denote
the set $\{0,1,\cdots,n-1\}$. Consider a sequence $S=(a_{0},a_{1},\cdots,a_{t-1})$
drawn from $[n]$. A \emph{cyclic variant} of $S$ is a sequence $(a_{i+1},a_{i+2},\cdots,a_{i})$,
where $0\leq i<t$, and subscripts are taken modulo $t$. We say $S$
is \emph{cyclic }if there is at most one subscript $i$ such that
$a_{i}>a_{i+1}$ (subscripts taken modulo $t$). Equivalently, $S$
is cyclic if and only if at least one of its cyclic variants is non-decreasing
$a_{i+1}\leq a_{i+2}\leq\cdots\leq a_{i}$. We say $S$ is \emph{anti-cyclic
}if there is at most one subscript $i$ such that $a_{i}<a_{i+1}$.
Equivalently, $S$ is anti-cyclic if and only if at least one of its
cyclic variants is non-increasing $a_{i+1}\ge a_{i+2}\ge\cdots\ge a_{i}$.
We shall say that $S$ is \emph{oriented }if $S$ is cyclic or $S$
is anti-cyclic. Orientation, that is, being cyclic or anti-cyclic,
is a property inherited by subsequences of an oriented sequence. We
say that $S$ is \emph{uniquely oriented }if $S$ is cyclic and not
anti-cyclic or anti-cyclic and not cyclic. Let $|S|$ denote the number
of distinct elements of $S$, that is, $|S|=|\{a_{0},\dots,a_{t-1}\}|$.
It is easy to see that if $|S|\le2$ then $S$ is either both cyclic
and anti-cyclic (for example, $S=(0,1)$) or neither (for example,
$S=(0,1,0,1)$). As we will show below, non-trivial and useful examples
of oriented sequences are those with $|S|=3$ and $|S|=4$.

We introduce an equivalence relation $\sim$ on the set of uniquely
oriented sequences whereby $S\sim T$ if $S$ and $T$ have the same
orientation. For example, if $S'$ is a cyclic variant of $S$, then
$S\sim S'$.

We write $\mathcal{T}_{n}$ for the full transformation semigroup
on $[n]$. For $\alpha\in\mathcal{T}_{n}$ and a sequence $S=(a_{0},a_{1},\cdots,a_{t-1})$
with entries drawn from $[n]$ we shall write $S\alpha$ for the sequence
$(a_{0}\alpha,a_{1}\alpha,\cdots,a_{t-1}\alpha)$. 
\begin{defn}
A mapping $\alpha\in\mathcal{T}_{n}$ is \emph{orientation-preserving}
(resp. \emph{orientation-reversing}) if the sequence $(0,1,\cdots,n-1)\alpha$
is cyclic (resp. anti-cyclic); these names are justified by Lemma
\ref{lem:OP-via-oriented}. The collection of all orientation-preserving
(resp. orientation-reversing) mappings in $\mathcal{T}_{n}$ is denoted
by $\mathcal{OP}_{n}$ (resp. $\mathcal{OR}_{n}$), while $\mathcal{P}_{n}$
is defined as $\mathcal{P}_{n}=\mathcal{OP}_{n}\cup\mathcal{OR}_{n}$. 
\end{defn}

If $S$ is a sequence, we denote the reversed sequence by $S^{R}$.
\begin{lem}
\label{lem:OP-via-oriented}Let $\alpha\in\mathcal{OP}_{n}$ (resp.
$\alpha\in\mathcal{OR}_{n}$) and let $S$ be an oriented sequence
such that $|S\alpha|\ge3$. Then $S\alpha\sim S$ (\textup{resp. $\ensuremath{S\alpha\sim S^{R}}$}).
\end{lem}

Since a sequence $S$ is cyclic (resp. anti-cyclic) if and only if
$S^{R}$ is anti-cyclic (resp. cyclic), it follows that we may equally
define $\mathcal{OP}_{n}$ as the set of all members of $\mathcal{T}_{n}$
that map anti-cyclic sequences to anti-cyclic sequences. 

That $\mathcal{OP}_{n}$ is a semigroup is now easily proved (Lemma
2.1 of \cite{CH1999}). However $\mathcal{OR}_{n}$ is not a semigroup:
$\mathcal{OR}_{n}\cdot\mathcal{OR}_{n}=\mathcal{OP}_{n}$, $\mathcal{OR}_{n}\cdot\mathcal{OP}_{n}=\mathcal{OP}_{n}\cdot\mathcal{OR}_{n}=\mathcal{OR}_{n}$,
and $\mathcal{OP}_{n}\cap\mathcal{OR}_{n}=\{\alpha\in\mathcal{P}_{n}:\,|\text{im}(\alpha)|\leq2\}$.
It follows that $\mathcal{P}_{n}$ is a semigroup. 

\section{Describing $\mathcal{OP}_{n}$ and $\mathcal{OR}_{n}$ with oriented
triples}

This result was stated but not proved in \cite[Proposition 1.1]{LM2006},
so here it is proved for the first time.
\begin{thm}
\label{thm:triples}Let $\alpha\in\mathcal{T}_{n}$. Then $\alpha\in\mathcal{OP}_{n}$
(resp. $\alpha\in\mathcal{OR}_{n})$ if and only if for every triple
$S=(i,j,k)$ of members of $[n]$, $S\alpha$ has the same (resp.
the opposite) orientation as $S$.
\end{thm}

\begin{proof}
($\Rightarrow)$ Follows from inheritance of orientation by subsequences.

($\Leftarrow)$ First, assume $\alpha\not\in\mathcal{OP}_{n}$. We
find a cyclic triple $(r,s,t)$ such that $(r\alpha,s\alpha,t\alpha)$
is not cyclic. Since $\alpha$ is not orientation-preserving there
exist distinct integers $i,j$ such that $i\alpha>(i+1)\alpha$ and
$j\alpha>(j+1)\alpha$.

\uline{Case 1}: $i\alpha\neq j\alpha$. Without loss we assume
that $i\alpha>j\alpha$. Put $r=i,s=j,t=j+1$. Then $(r,s,t)$ is
cyclic but since $i\alpha>j\alpha>(j+1)\alpha$, $(r\alpha,s\alpha,t\alpha)$
is anti-cyclic (and not cyclic as the three entries are pairwise distinct).

\uline{Case 2}: $(i+1)\alpha\neq(j+1)\alpha.$ Without loss we
assume that $(i+1)\alpha>(j+1)\alpha$. Put $r=i,s=i+1,t=j+1$. Then
$(r,s,t)$ is cyclic but since $i\alpha>(i+1)\alpha>(j+1)\alpha$,
$(r\alpha,s\alpha,t\alpha)$ is anti-cyclic.

\uline{Case 3}: $i\alpha=j\alpha$ and $(i+1)\alpha=(j+1)\alpha$.
Since $|\text{im\ensuremath{\alpha}\ensuremath{|\geq3} }$there exists
$k$ such that $k\alpha\not\in\{i\alpha,(i+1)\alpha\}$. It follows
that the members of the list $i,i+1,k,j,j+1$ are pairwise distinct
and, by interchanging the symbols $i$ and $j$ if necessary, we may
assume that $(i,i+1,k,j,j+1)$ is cyclic. 

(1) Suppose that $k\alpha>i\alpha.$ Put $r=k,s=i,t=i+1$. Then $(r,s,t)$
is cyclic but since $k\alpha>i\alpha>(i+1)\alpha$, $(r\alpha,s\alpha,t\alpha)$
is anti-cyclic.

(2) Suppose $i\alpha>k\alpha>(i+1)\alpha$. Put $r=i,s=k,t=j+1$.
Then $(r,s,t)$ is cyclic but $(r\alpha,s\alpha,t\alpha)$ is anti-cyclic.

(3) Suppose $(i+1)\alpha>k\alpha$. Put $r=i,s=i+1,t=k$. Then $(r,s,t)$
is cyclic but since $i\alpha>(i+1)\alpha>k\alpha$, $(r\alpha,s\alpha,t\alpha)$
is anti-cyclic. 

This completes a proof of the reverse implication in the case where
$\alpha\not\in\mathcal{OP}_{n}$. For the case where $\alpha\not\in\mathcal{OR}_{n}$
let $\gamma$ be the permutation of order reversal: $i\gamma=n-1-i$
$(0\leq i\leq n-1)$. Then $\gamma\in\mathcal{OR}_{n}$ so that $\alpha\gamma\not\in\mathcal{OP}_{n}$
as otherwise $\alpha\gamma^{2}=\alpha\in\mathcal{OR}_{n}$. We then
apply the previous argument to conclude that there exists three pairwise
distinct integers forming a cyclic triple $(r,s,t)$ such that $(r\alpha\gamma,s\alpha\gamma,t\alpha\gamma)$
is anti-cyclic. But then $(r\alpha\gamma^{2},s\alpha\gamma^{2},t\alpha\gamma^{2})=(r\alpha,s\alpha,t\alpha)$
is cyclic, and noting this completes the proof of (b). 
\end{proof}

\section{Describing $\mathcal{P}_{n}$ with oriented quadruples and with chord
intersections}

We introduce a new characterization of $\mathcal{P}_{n}$ in terms
of the preservation of a single geometric property. This stems from
the study of \emph{Gauss diagrams} in knot theory, which consist of
a circle and chords. The paper \cite{D1936} first described Gauss
diagrams satisfying an important property known as being \emph{realizable},
and now we know that one can describe realizable Gauss diagrams by
only using the information stating, for every pair of chords, whether
or not they intersect \cite{STZ2009}. 
\begin{defn}
Consider the integers $0,1,\cdots,n-1$ positioned clockwise around
a circle. A mapping $\alpha\in\mathcal{T}_{n}$ has the \emph{chord
property} if whenever the chords $ac$ and $bd$ intersect $(a,b,c,d\in[n])$
then so do the chords $a\alpha c\alpha$ and $b\alpha d\alpha$. (See
Figure \ref{fig:chords}.)
\end{defn}

\begin{figure}
\includegraphics[scale=0.5]{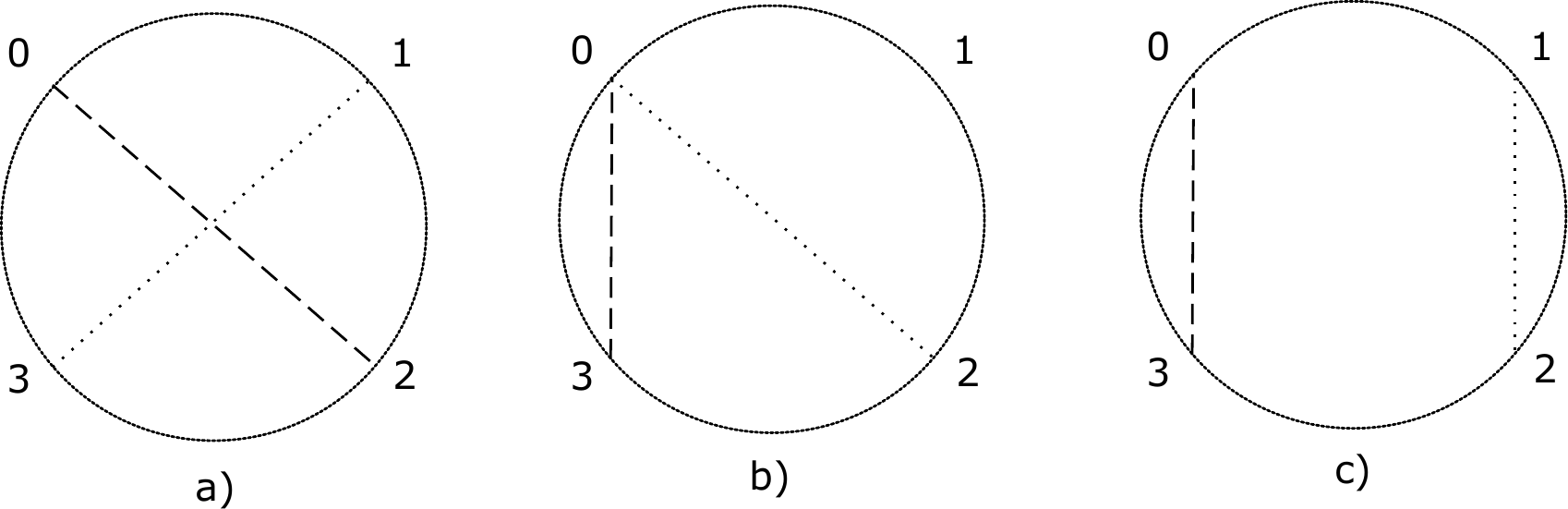}

\caption{\label{fig:chords}Diagram a) shows two intersecting chords $13$
and $02$ as a dotted line and dashed line. If we apply a mapping
$0\protect\mapsto0$, $1\protect\mapsto0$, $2\protect\mapsto3$,
$3\protect\mapsto2$, the rearranged chords still intersect, see diagram
b). If, instead, we apply a mapping $0\protect\mapsto0$, $1\protect\mapsto1$,
$2\protect\mapsto3$, $3\protect\mapsto2$, the rearranged chords
do not intersect, see diagram c).}
\end{figure}

This definition makes allowance for one point chords. If there is
repetition among the $a,b,c,d$ due to the chords having a common
endpoint, then chord intersection is preserved by any mapping $\alpha$
by virtue of the image of that common endpoint. We shall denote the
clockwise arc of the circle $R$ that runs from $a$ to $c$ as $\overrightarrow{ac}$. 
\begin{lem}
\textbf{\label{lem:chords}}Let $a,b,c,d\in[n]$ be four points on
the circumference of a circle on which the members of $[n]$ are placed
clockwise. The chords $ac$ and $bd$ meet if and only if the quadruple
$S=(a,b,c,d)$ is oriented. 
\end{lem}

\begin{proof}
Suppose the chords $ac$ and $bd$ meet. Suppose that our four points
are distinct. If $b$ lies on the arc $\overrightarrow{ac}$ then
$d$ lies on $\overrightarrow{ca}$ and $(a,b,c,d)$ is oriented clockwise
on the circle, whence $(a,b,c,d)$ is cyclic. Alternatively $b$ lies
on the arc $\overrightarrow{ca}$, and $d$ does not, in which case
$d$ lies on $\overrightarrow{ac}$ and $(a,d,c,b)$ is cyclic. But
then $(a,d,c,b)^{R}=(b,c,d,a)\sim(a,b,c,d)$ is anti-cyclic. In either
event, $S$ is oriented.

On the other hand, suppose that $S$ has a repeated entry. If $a=c$,
then since $ac$ meets $bd$, it follows that $a=b=c$ or $a=d=c$,
in which case $S$ is oriented; the same conclusion follows if $b=d$.
Otherwise two cyclically adjacent entries in $S$, for instance $a$
and $b$, are equal, in which case it also follows that $S$ is oriented. 

Conversely, if $(a,b,c,d)$ is cyclic, then $b$ lies on $\overrightarrow{ac}$
and $d$ lies on $\overrightarrow{ca}$ so that the chords $ac$ and
$bd$ meet. On the other hand if $(a,b,c,d)$ is anti-cyclic then
$(a,b,c,d)^{R}=(d,c,b,a)$ is cyclic and so the chord $db=bd$ meets
the chord $ca=ac$, as required. 
\end{proof}
\begin{prop}
\textbf{\label{prop:chords}}A mapping $\alpha\in\mathcal{T}_{n}$
has the chord property if and only if for every oriented quadruple
$S=(a,b,c,d$), the image sequence $S\alpha$ is also oriented. 
\end{prop}

\begin{proof}
Suppose that $\alpha$ has the chord property and let $S=(a,b,c,d$)
be an oriented sequence. Let $S\alpha=(A,B,C,D)$. By Lemma \ref{lem:chords},
the chords $ac$ and $bd$ intersect, and hence by hypothesis so do
the chords $AC$ and $BD$. Again by Lemma \ref{lem:chords} it follows
that $S\alpha$ is oriented. 

Conversely, suppose that the image of every oriented quadruple under
$\alpha$ is oriented and let chords $ac$ and $bd$ meet. By Lemma
\ref{lem:chords} $(a,b,c,d)$ is oriented and so by hypothesis so
is $(A,B,C,D)$. Again by Lemma \ref{lem:chords}, it follows that
chords $AC$ and $BD$ meet, whence we conclude that $\alpha$ has
the chord property. 
\end{proof}
\begin{thm}
(a) $\alpha\in\mathcal{P}_{n}$ if and only if for every oriented
quadruple $S=(a,b,c,d)$, $S\alpha$ is also oriented. 

(b) $\alpha$ has the chord property if and only if $\alpha\in\mathcal{P}_{n}$. 
\end{thm}

\begin{proof}
(b) follows from Proposition \ref{prop:chords} together with (a).

(a)($\Rightarrow)$ Follows from inheritance of orientation by subsequences.

(a)($\Leftarrow$) We again argue the contrapositive, like in Theorem
\ref{thm:triples}, so suppose that $\alpha\not\in\mathcal{P}_{n}$.
Let $m=$ min$(\text{im\ensuremath{(\alpha))}}.$ Choose $i$ such
that $i\alpha=m$ but $i\alpha<(i+1)\alpha$. Consider the cyclic
sequence $I:i+1,i+2,\cdots,i+n-2$ and let $j$ be the first listed
member of $I$ such that $j\alpha>(j+1)\alpha$; since $\alpha\not\in\mathcal{OP}_{n}$,
$j$ occurs in $I$, as there are at least $2$ integers $t$ such
that $t\alpha>(t+1)\alpha$. Similarly let $M=$ max$(\text{im\ensuremath{(\alpha))}}.$
Choose $i'$ such that $i'\alpha=M$ but $i'\alpha>(i'+1)\alpha$.
Consider the cyclic sequence $I':i'+1,i'+2,\cdots,i'+n-2$ and let
$j'$ be the first listed member of $I'$ such that $j'\alpha<(j'+1)\alpha$;
since $\alpha\not\in\mathcal{OR}_{n}$, $j'$ occurs in $I'$ as there
are at least $2$ integers $t$ such that $t\alpha<(t+1)\alpha$.
Note that $(i+1)\alpha\leq j\alpha$ and $(i'+1)\alpha\geq j'\alpha.$

\uline{Case 1}: $(i+1)\alpha=j\alpha$ or $(i'+1)\alpha=j'\alpha$.
First suppose $(i+1)\alpha=j\alpha$, whence by definition of $j$
and the assumption that $(i+1)\alpha=j\alpha$ it follows that

\begin{equation}
m=i\alpha<(i+1)\alpha=\cdots=j\alpha>(j+1)\alpha.\label{eq:inequality1}
\end{equation}
Consider the subsequence $J$ of $I$ given by $J:j+1,j+2,\cdots,i+n-2$.
Since $\alpha\not\in\mathcal{OR}_{n}$, there exists $k$ in $J$
such that $k\alpha<(k+1)\alpha$, (for otherwise the first such $k$
is $k=i+n-1$, but then $k+1=i$, contradicting $(k+1)\alpha>k\alpha\geq i\alpha$).
By choosing the first listed such $k$ in $J$ we have by (\ref{eq:inequality1})
that $j\alpha>k\alpha$ and so
\begin{equation}
m=i\alpha<(i+1)\alpha=j\alpha>k\alpha<(k+1)\alpha>i\alpha=m.\label{eq:inequality2}
\end{equation}
From (\ref{eq:inequality2}) we infer that $S=(i,i+1,k,k+1)$ is a
cyclic quadruple of pairwise distinct integers but, by (\ref{eq:inequality2}),
$S\alpha$ is not oriented. 

Similarly if $(i'+1)\alpha=j'\alpha$ the dual argument obtained by
reversing all inequalities and interchanging $\mathcal{OP}_{n}$ and
$\mathcal{OR}_{n}$ throughout yields a cyclic sequence $S'=(i',i'+1,k',k'+1)$
such that $S'\alpha$ is not oriented as 
\[
M=i'\alpha>(i'+1)\alpha=j'\alpha<k'\alpha>(k'+1)\alpha<i'\alpha=M.
\]

\uline{Case 2}: $(i+1)\alpha<j\alpha$ and $(i'+1)\alpha>j'\alpha$.
From the given inequalities we now have 
\[
m=i\alpha<(i+1)\alpha<M=i'\alpha>(i'+1)\alpha>m=i\alpha,
\]
which implies that $T=(i,i+1,i',i'+1)$ is a cyclic sequence of pairwise
distinct integers such that $T\alpha=(i\alpha,(i+1)\alpha,i'\alpha,(i'+1)\alpha)$
is not oriented.
\end{proof}

\end{document}